\documentclass{amsart}
\usepackage{tabularx}
\usepackage{enumitem,amsmath,amsfonts,amssymb,xcolor,latexsym,euscript,mathrsfs,hyperref}
\usepackage[numbers]{natbib}

\newtheorem{lemma}{\bf Lemma}[section]

\newtheorem{prop}[lemma]{\bf Proposition}
\newtheorem{thm}[lemma]{\bf Theorem}

\DeclareMathOperator{\Sym}{Sym}

\input xy
\xyoption{all}

\title[]{Bounds in terms of the number of cyclic subgroups}

\author{Xiaofang Gao} 
\address{Xiaofang Gao. Departamento de Matem\'atica, Universidade de Bras\'ilia, Campus 
Universit\'ario \\ Darcy Ribeiro, Bras\'ilia-DF, 70910-900, Brazil. \newline
ORCID:  https://orcid.org/0000-0001-8106-7941}
\email{gaoxiaofang2020@hotmail.com}

\author{Martino Garonzi}
\address{Martino Garonzi. Departamento de Matem\'atica, Universidade de Bras\'ilia, Campus 
Universit\'ario Darcy Ribeiro, Bras\'ilia-DF, 70910-900, Brazil. \newline
ORCID: https://orcid.org/0000-0003-0041-3131}
\email{martino@mat.unb.br, mgaronzi@gmail.com}

\thanks{The first author acknowledges the support of the CAPES PhD
fellowship and the NSF of China - Grant number 12161035. 
The second author acknowledges the support of the Funda\c{c}\~{a}o de Apoio \`a Pesquisa do Distrito Federal (FAPDF) -- demanda espont\^{a}nea 09/2023 and of the Conselho Nacional 
de Desenvolvimento Cient\'ifico e Tecnol\'ogico (CNPq), Universal
- Grant number 402934/2021-0.}

\date{}


\begin{document}

\begin{abstract}
A family of groups is called (maximal) cyclic bounded ((M)CB) if, for every natural number $n$, there are only finitely many groups in the family with at most $n$ (maximal) cyclic subgroups. We prove that the family of groups of prime power order is MCB. We also prove that the family of finite groups without cyclic coprime direct factors is CB. As a consequence, a natural number $n \geqslant 10$ is prime if and only if there are only finitely many finite noncyclic groups with precisely $n$ cyclic subgroups.
\end{abstract}

\maketitle

\section{Introduction}
In this paper, $G$ will always denote a finite group. A \textit{covering} of $G$ is a family $\mathscr{H} = \{H_1,\ldots,H_k\}$ of proper subgroups of $G$ whose union is $G$. Note that $G$ admits a covering if and only if $G$ is noncyclic. The covering $\mathscr{H}$ of $G$ is called \textit{irredundant} if $\mathscr{H} \setminus \{H_i\}$ is not a covering for all $i \in \{1,\ldots,k\}$. In other words, $\mathscr{H}$ is irredundant if and only if no proper subfamily of $\mathscr{H}$ is a covering.

A subgroup $H$ of $G$ is called a \textit{maximal cyclic subgroup} if it is cyclic and it is not properly contained in any cyclic subgroup of $G$. It is easy to see that, if $G$ is noncyclic, then the maximal cyclic subgroups of $G$ form an irredundant covering. Rog\'erio \cite{Rogerio} defined $\lambda(G)$ to be the maximal size of an irredundant covering of $G$. This function has received attention in recent years, see \cite{BLR,GL,Rogerio}. By \cite[Proposition 4]{Rogerio}, if $G$ is noncyclic, then $\lambda(G)$ equals the number of maximal cyclic subgroups of $G$. In the present paper, if $G$ is a cyclic group, we set $\lambda(G)=1$, so that $\lambda(G)$ always coincides with the number of maximal cyclic subgroups of $G$. Using the fact that every cyclic subgroup is contained in a maximal cyclic subgroup (by finiteness of $G$), it is easy to see that a group $G$ is cyclic if and only if $\lambda(G)=1$.

If $X$ is an arbitrary subset of $G$, define 
\begin{equation} \label{formulaEuler}
    c(X) := \sum_{x \in X} \frac{1}{\varphi(o(x))}
\end{equation}
where $o(x)$ denotes the order of $x$ and $\varphi$ is Euler's totient function. Since a cyclic group of order $n$ has $\varphi(n)$ generators, the number of cyclic subgroups of $G$ is precisely equal to $c(G)$. More in general, if whenever $x \in X$ and $k$ is an integer coprime to $o(x)$ we have $x^k \in X$, then $c(X)$ equals the size of the set $\{\langle x \rangle\ :\ x \in X\}$, where $\langle x \rangle$ is the cyclic subgroup generated by $x$.

We say that an element $x \in G$ is \textit{primitive} if $\langle x \rangle$ is a maximal cyclic subgroup of $G$. Equivalently, whenever $x$ is a power of an element $y$, the element $y$ is a power of $x$. Let $P$ be the set of primitive elements of $G$. It is easy to show that $G$ has precisely $c(P)$ maximal cyclic subgroups, in other words $\lambda(G)=c(P)$. For example, if $G$ is a $p$-group (for some prime number $p$) then $\lambda(G) = c(G)-c(G^p)$, where $G^p = \{g^p\ :\ g \in G\}$ is the set of non-primitive elements.

Assume $A$ and $B$ are groups of coprime orders. Then $c(A \times B) = c(A) \cdot c(B)$ and $\lambda(A \times B) = \lambda(A) \cdot \lambda(B)$. This can also be proved by using (\ref{formulaEuler}). Specifically, in the case of $\lambda$, an element $(a,b) \in A \times B$ is primitive in $A \times B$ if and only if $a$ is primitive in $A$ and $b$ is primitive in $B$ (see Lemma \ref{primcoprime}).

We give some examples. If $G$ is a cyclic $p$-group of order $p^n$ (where $p$ is a prime number), then $c(G)=n+1$ and $\lambda(G)=1$. As an easy but less trivial example, consider $G=D_{2n}$, the dihedral group of order $2n$. The maximal cyclic subgroups of $G$ are the non-normal subgroups of order $2$ and the unique cyclic subgroup of order $n$, so that $\lambda(G)=n+1$. On the other hand $c(G) = n + d(n)$ where $d(n)$ is the number of positive divisors of $n$. If $m \geqslant 4$ is an integer, there exists a noncyclic (abelian) finite group $G$ with $c(G)=m$. Indeed we can write $m-2=np$ for some prime $p$ and some integer $n \geqslant 1$ and $c(C_p \times C_{p^n}) = np+2 = m$ by Lemma \ref{lambda-abelian}.

A natural question is the following: what does $\lambda(G)$ bound? Using a result of L. Pyber about noncommuting elements, we prove that the index of the center $|G:Z(G)|$ can be bounded above in terms of $\lambda(G)$ (see Proposition \ref{noncommuting}). In particular, the Fitting length and the derived length of a solvable group can be bounded above in terms of $\lambda(G)$. The same is true for the nilpotency class of a nilpotent group. Concerning the derived length, we prove the following.

\begin{thm} \label{solbound}
If $G$ is any finite solvable group then the derived length of $G$ is at most $2+\frac{5}{2} \log_3(\lambda(G))$.
\end{thm}

We say that a family of groups is \textit{(maximal) cyclic bounded ((M)CB)} if, for every natural number $n$, there are only finitely many groups $G$ in the family (up to isomorphism) with at most $n$ (maximal) cyclic subgroups. In other words, a family $\mathscr{F}$ of groups is CB (resp. MCB) if, for every natural number $n$, there are only finitely many groups $G$ in $\mathscr{F}$ (up to isomorphism) such that $c(G) \leqslant n$ (resp. $\lambda(G) \leqslant n$). Since $\lambda(G) \leqslant c(G)$, every MCB family is also CB. However, for example, the family of cyclic $2$-groups is CB but not MCB.

We prove the following.

\begin{thm} \label{MCB}
The family of noncyclic groups of prime power order is MCB.
\end{thm}
More precisely, we prove that if $G$ is a noncyclic finite $p$-group and $t=\lambda(G)$ then $|G| \leqslant c^t \cdot t^{t^2}$ for some constant $c$.

\begin{thm} \label{centerless}
The family of groups $G$ such that every nontrivial Sylow subgroup of the center $Z(G)$ is noncyclic is MCB.
\end{thm}

In particular, the family of groups with trivial center is MCB.

We say that a group $H$ is a \textit{coprime direct factor} of a group $G$ if $G$ is isomorphic to a direct product $H \times K$ and $|H|$, $|K|$ are coprime. Note that, in this situation, if $H$ is a cyclic $p$-group of order $p^h$ then $c(G) = (h+1) \cdot c(K)$ and this shows that the family of all finite groups is far from being CB (we have no control on $p$). However, we prove the following result.

\begin{thm} \label{CB}
The family of groups without cyclic coprime direct factors is CB.
\end{thm}

Observe that the family of groups without cyclic coprime direct factors is not MCB. As an example of this, consider $G = G_k = C_3 \rtimes C_{2^k}$ with the action given by inversion. Then $G$ is not isomorphic to a direct product of two nontrivial groups and $\lambda(G)=4$. Indeed, writing $G=\langle a, b: a^3=b^{2^k}=1, a^b=a^{-1}\rangle$, since $a^{b^2}=a$, we have $N:=\langle a, b^2\rangle=\langle a\rangle \times \langle b^2\rangle\cong C_{3 \cdot 2^{k-1}}$. The subgroup $N$ is maximal cyclic and $\langle x \rangle$ is maximal cyclic for every $x \in G-N$, since $o(x)=2^k$. Therefore 
\begin{align*}
\lambda(G) & = 1 + |G-N|/\varphi(2^k) = 4, \\
c(G) & = c(N)+c(G-N) = c(C_3) \cdot c(C_{2^{k-1}}) + |G-N|/\varphi(2^k)=2k+3.
\end{align*}
This also shows that the family of groups $\{G_k\ :\ k \geqslant 1\}$ is CB but not MCB.

Let $\mathscr{B}$ be the set of positive integers $n$ such that there are only finitely many finite noncyclic groups with precisely $n$ cyclic subgroups. From the work of Ashrafi and Haghi \cite{AH,HA2018} (see also \cite{Kalra} and \cite{Zhou}) we immediately deduce that $3,4,5,6,7,9 \in \mathscr{B}$ and $8,10 \not \in \mathscr{B}$. These papers have as objective to classify the groups with a given number of cyclic subgroups. It is interesting to observe that, in all known results, there are only finitely many noncyclic groups whose number of cyclic subgroups is a given prime number, in other words it seems that $\mathscr{B}$ contains all the prime numbers. We confirm this in the following theorem.

\begin{thm} \label{setB}
$\mathscr{B} = \{1,4,6,9\} \cup \mathscr{P}$ where $\mathscr{P}$ is the set of prime numbers.
\end{thm}

\section{Proof of Theorem \ref{solbound}}

Let $\mathscr{C}$ be the family of all maximal cyclic subgroups of $G$, so that $|\mathscr{C}|=\lambda=\lambda(G)$. Of course $G$ acts on $\mathscr{C}$ by conjugation. The kernel $N$ of this action is the intersection of the normalizers of all maximal cyclic subgroups of $G$, therefore $G/N$ is isomorphic to a subgroup of $\Sym(\lambda)$.

We claim that $N$ is actually equal to the set of elements of $G$ that normalize every subgroup of $G$. To see this, let $g \in G$ be an element that normalizes every maximal cyclic subgroup of $G$. If $H$ is any subgroup of $G$ and $h \in H$, then there exists a maximal cyclic subgroup $\langle x \rangle$ of $G$ containing $\langle h \rangle$, and by assumption $x^g = x^k$ for some integer $k$. Since $h \in \langle x \rangle$, there is some integer $t$ such that $h=x^t$, so that
$$h^g = (x^t)^g = (x^g)^t = (x^k)^t = (x^t)^k = h^k \in \langle h \rangle \leqslant H,$$
this proves that $g$ normalizes $H$. This proves the claim.

The intersection of the normalizers of the subgroups of $G$ has been studied by Schenkman in  \cite{Schenkman}. Using their results, we deduce that $N$ is contained in the second term of the upper central series, $N \leqslant Z_2(G)$, so $N$ is nilpotent of class at most $2$. In particular $N''=\{1\}$.
By \cite[Theorem 1]{Dixon}, the derived length of a solvable subgroup of $\Sym(\lambda)$ is at most $\frac{5}{2} \log_3(\lambda)$. Since $N$ is metabelian and $G/N$ is isomorphic to a subgroup of $\Sym(\lambda)$, we deduce that, if $G$ is solvable, then the derived length of $G$ is at most $2+\frac{5}{2} \log_3(\lambda)$. This proves Theorem \ref{solbound}.

\section{Proof of Theorem \ref{MCB}}

In this section, we prove Theorem \ref{MCB}. First, we need some preliminary results.

\begin{prop}\label{subquo}
If $H \leqslant G$ and $N \unlhd G$ then $\lambda(H) \leqslant \lambda(G)$ and $\lambda(G/N) \leqslant \lambda(G)$.
\end{prop}

\begin{proof}
Let $\langle h_i \rangle$, $i=1,\ldots,n$, be the maximal cyclic subgroups of $H$, where $n=\lambda(H)$, and let $\langle x_i \rangle$, $i=1,\ldots,n$, be maximal cyclic subgroups of $G$ such that $\langle h_i \rangle \leqslant \langle x_i \rangle$ for all $i$. To prove that $\lambda(H) \leqslant \lambda(G)$ it is enough to prove that, if $\langle x_i \rangle = \langle x_j \rangle$, then $\langle h_i \rangle = \langle h_j \rangle$. We can write $h_i = x^a$ and $h_j = x^b$ for some positive integer $a,b$, where $x=x_i$. Let $d$ be the greatest common divisor of $a,b$, then we can write $dr=a$ and $ds=b$ with $r,s \in \mathbb{N}$. Moreover $ua+vb=d$ for some integer $u,v$ and hence $y = x^d = (x^a)^u (x^b)^v \in H$. Since $y^r=h_i$ and $y^s=h_j$, we deduce that both $\langle h_i \rangle$ and $\langle h_j \rangle$ are contained in $\langle y \rangle$. Since they are maximal cyclic subgroups of $H$ and $y \in H$, we deduce that they are both equal to $\langle y \rangle$. Therefore $\langle h_i \rangle = \langle y \rangle = \langle h_j \rangle$.

Let $N \unlhd G$. We need to prove that $\lambda(G/N) \leqslant \lambda(G)$. If $G/N$ is cyclic then $\lambda(G/N)=1$ and there is nothing to prove, so now assume that $G/N$ is noncyclic, so that $G$ is noncyclic as well. If $\{H_1/N,\ldots,H_k/N\}$ is an irredundant covering of $G/N$ of size $k=\lambda(G/N)$, then $\{H_1,\ldots,H_k\}$ is an irredundant covering of $G$ of size $k$. Therefore $\lambda(G/N) \leqslant \lambda(G)$.
\end{proof}

\begin{lemma} \label{cyclicfactor}
Let $p$ be a prime number, let $A,B$ be nontrivial finite $p$-groups with $A$ cyclic and let $G:=A \times B$. Then $A \times \{1\}$ is a maximal cyclic subgroup of $G$.
\end{lemma}
\begin{proof}
If this is not true, then writing $A=\langle a \rangle$, there exists an element $(a^i,b) \in G$ such that $\langle (a,1) \rangle$ is properly contained in $\langle (a^i,b) \rangle$. This means that $b \neq 1$ and there exists an integer $m$ such that $a^{im}=a$, $b^m=1$. The fact that $a^{im}=a$ implies that $o(a)$, which is a power of $p$ different from $1$ (as $A$ is a nontrivial $p$-group), divides $im-1$, implying that $p$ does not divide $m$. Since $b^m=1$ and $B$ is a $p$-group, we deduce that $b=1$, a contradiction.
\end{proof}

\begin{lemma} \label{lambda-abelian}
Let $p$ be a prime number, let $P$ be a finite nontrivial $p$-group and let $c$ be the number of cyclic subgroups of $P$, let $\lambda$ be the number of maximal cyclic subgroups of $P$. If $G= P\times C_p$, then
$$\begin{array}{l}
c(G) = p(c-1)+2, \\
\lambda(G) = (p-1)(c-1)+\lambda+1.
\end{array}$$
In particular, $c(C_{p^a} \times C_p) = ap+2$ and  $\lambda(C_{p^a} \times C_p) = ap-a+2$.
\end{lemma}
\begin{proof}
For every $b \in \mathbb{N}$, let $c_b$ be the number of elements of order $p^b$ in $P$. If $b \geqslant 2$, then $G$ has $p \cdot c_b$ elements of order $p^b$. Moreover, $G$ has $p \cdot c_1+p-1$ elements of order $p$ and of course one element of order $1$. Write $|P|=p^a$. It follows that
$$c(G) = 1+\frac{p \cdot c_1+p-1}{\varphi(p)} + \sum_{b=2}^a \frac{p \cdot c_b}{\varphi(p^b)} = p \cdot \sum_{b=1}^a \frac{c_b}{\varphi(p^b)} +2 = p(c-1)+2.$$
Now, the $c-\lambda$ non-maximal cyclic subgroups of $P$ are not maximal cyclic subgroups of $G$, therefore
$$\lambda(G) = p(c-1)+2-(c-\lambda) = (p-1)c-p+2+\lambda.$$
This is because, if $K=\langle x \rangle$ is any cyclic subgroup of $G$, then $x^p \in P \times \{1\}$, therefore all the subgroups of $G$ properly contained in $K$ are subgroups of $P \times \{1\}$. Therefore all the cyclic subgroups of $G$ not contained in $P \times \{1\}$ are maximal cyclic.
\end{proof}

Let $n$ be the maximal size of a set of pairwise non-commuting elements of a group $G$. We claim that $n \leqslant \lambda(G)$. To see this, let $S=\{x_1,\ldots,x_n\}$ be a set of pairwise non-commuting elements of maximal size. Then, $G$ has an element $y_i$ such that $x_i \in \langle y_i \rangle$, and $\langle y_i \rangle$ is a maximal cyclic subgroup of $G$ for all $i$. We claim that the $\langle y_i \rangle$'s are pairwise distinct. Indeed, assume that $C = \langle y_i \rangle = \langle y_j \rangle$ for some $i \neq j$. Then the two distinct elements $x_i,x_j$ belong to the same cyclic subgroup $C$, hence they commute, a contradiction.

L. Pyber in \cite{Pyber} proved that $|G:Z(G)| \leqslant \alpha^n$ for some absolute constant $c$, so that
$$|G:Z(G)| \leqslant \alpha^n \leqslant \alpha^{\lambda(G)}.$$In other words, we have the following.
\begin{prop} \label{noncommuting}
There exists a constant $\alpha$ such that $|G:Z(G)| \leqslant \alpha^{\lambda(G)}$ for every finite group $G$.
\end{prop}

We now prove Theorem \ref{MCB}, i.e. the fact that the family of finite noncyclic groups of prime power order is MCB. More precisely, we prove that there exists a constant $c$ such that, if $G$ is a noncyclic finite $p$-group and $t=\lambda(G)$, then $|G| \leqslant c^t \cdot t^{t^2}$.

Let $G$ be a noncyclic finite $p$-group, where $p$ is a fixed prime number. If $G$ is abelian then we can write $G = \prod_{i=1}^k C_{p^{a_i}} = \prod_{i=1}^k \langle g_i \rangle$ with $k \geqslant 2$. Since the $\langle g_i \rangle$ are maximal cyclic subgroups of $G$ by Lemma \ref{cyclicfactor}, we have $k \leqslant \lambda(G)$, so we are left to bound the $a_i$'s. Since $\lambda(H) \leqslant \lambda(G)$ for $H \leqslant G$ by Proposition \ref{subquo}, we may assume that $G=C_{p^a} \times C_p = \langle x \rangle \times \langle y \rangle$ hence, by Lemma \ref{lambda-abelian}, $\lambda(G)=ap-a+2 \geqslant a$ and similarly $\lambda(G) \geqslant p$, therefore $|G| \leqslant t^{t^2}$ where $t=\lambda(G)$.

Now assume $G$ is nonabelian. By Proposition \ref{noncommuting}, $|G:Z(G)|$ is bounded by a function of $\lambda(G)$.
If $H$ is a noncyclic abelian subgroup of $G$ then $K = \langle H,Z(G) \rangle$ is a noncyclic abelian subgroup of $G$ hence, by the abelian case, $|K|$ is bounded in terms of $\lambda(K) \leqslant \lambda(G)$, so since $Z(G)$ is contained in $K$, its order is bounded in terms of $\lambda(G)$. We are left to analyze the case in which all abelian subgroups of $G$ are cyclic. The class of $p$-groups all of whose abelian subgroups are cyclic is known: such a group is either cyclic or a generalized quaternion group (see \cite[Theorem 4.10 of Chapter 5]{Gorenstein}). Since $G$ is noncyclic, it must be a generalized quaternion group $Q_{2^n}$, therefore $|Z(G)|=2$ and we are done.

\section{Proof of Theorem \ref{centerless}}

Before proving Theorem \ref{centerless}, we need a lemma. Recall that an element $g \in G$ is called primitive if $\langle g \rangle$ is a maximal cyclic subgroup of $G$.

\begin{lemma} \label{primcoprime}
Let $A,B$ be groups of coprime orders and let $G=A \times B$. An element $(a,b) \in G$ is primitive in $G$ if and only if $a$ is primitive in $A$ and $b$ is primitive in $B$. As a consequence, $\lambda(G) = \lambda(A) \cdot \lambda(B)$.
\end{lemma}

\begin{proof}
Assume $(a,b) \in G$ is primitive. We prove that $a$ is primitive. Suppose $a=x^k$ for some $x \in A$, $k \in \mathbb{Z}$. We need to prove that $x$ is a power of $a$. Let $d$ be the GCD between $k$ and $|B|$. Since $k/d$ is coprime to $|B|$, we can write $b=y^{k/d}$ for a suitable $y \in \langle b \rangle$, hence $(a,b)=(x^d,y)^{k/d}$. Since $(a,b)$ is primitive, there exists $r \in \mathbb{Z}$ such that $(x^d,y)=(a,b)^r=(a^r,b^r)$, hence $x^d=a^r$. Since $d$ is coprime to $|A|$, this implies that $x$ is a power of $a$. This proves that $a$ is primitive and, similarly, $b$ is primitive too. Conversely, assume $a$ is primitive in $A$ and $b$ is primitive in $B$. Assume $(a,b) = (x,y)^k = (x^k,y^k)$, then $a=x^k$ and $b=y^k$. Since $a,b$ are primitive, there exist $l,m \in \mathbb{Z}$ such that $x=a^l$ and $y=b^m$. By the Chinese Remainder Theorem, there exists $r \in \mathbb{Z}$ such that $r \equiv l \mod |A|$ and $r \equiv m \mod |B|$, therefore $(a,b)^r = (a^r,b^r) = (a^l,b^m) = (x,y)$. This proves that $(a,b)$ is primitive.
\end{proof}

We now prove Theorem \ref{centerless}. Let $G$ be a finite group such that every nontrivial Sylow subgroup of $Z(G)$ is noncyclic. 
Write $Z(G)=P_1\times P_2\times \ldots\times P_k$ where $P_i$ is a nontrivial Sylow subgroup of $Z(G)$ for $i=1,2,\ldots, k$. Let $z:=\min\{\lambda(P_1), \lambda(P_2),\ldots, \lambda(P_k)\}$. Since $P_i$ is noncyclic for all $i$, we have $z \geqslant 2$. By Proposition \ref{subquo}, 
$$z^k \leqslant  \lambda(P_1) \cdot \ldots \cdot \lambda(P_k) = \lambda(Z(G)) \leqslant \lambda(G),$$ 
it follows that $k\leqslant \log_z(\lambda(G))$. By Theorem \ref{MCB}, the family of noncyclic groups of prime power order is MCB, so by Proposition \ref{subquo} we can bound the order of $P_i$ in terms of $\lambda(G)$, thus $|Z(G)|$ is bounded by a function of $\lambda(G)$. Therefore $|G|=|G:Z(G)| \cdot |Z(G)|$ is also bounded by a function of $\lambda(G)$ by Proposition \ref{noncommuting}.

\section{Proof of Theorem \ref{CB}}

Theorem \ref{MCB} implies that the family of noncyclic groups of prime power order is CB. More precisely, we have the following.

\begin{prop}\label{pgroup}
If $G$ is a noncyclic finite $p$-group and $t = c(G)$ then $|G|\leqslant t^t$.
\end{prop}

\begin{proof}
Let $G$ be a noncyclic group of order $p^n$ where $p$ is a prime, we need to bound $p$ and $n$ in terms of $t=c(G)$. If $p=2$ then $p<t$ since $G$ is a noncyclic group. If $p \geqslant 3$, then by \cite[Theorem 4.10 of Chapter 5]{Gorenstein}, $G$ has a subgroup isomorphic to $C_p \times C_p$, which has $p+2$ cyclic subgroups, therefore $p \leqslant t$. If $p^m$ is the exponent of $G$ and $g \in G$ has order $p^m$, then $\langle g \rangle$ has $m+1$ cyclic subgroups, therefore $m \leqslant t-1$ and $\exp(G) \leqslant t^{t-1}$. Since each cyclic subgroup of $G$ has at most $\varphi(p^m) = p^{m-1}(p-1)$ generators, $G$ has at least $\frac{p^n-1}{p^{m-1}(p-1)}$ cyclic subgroups, therefore
$$t\geqslant \frac{p^n-1}{p^{m-1}(p-1)}=\frac{p^{n-1}+p^{n-2}+\ldots+p+1}{p^{m-1}}\geqslant \frac{p^{n-1}}{p^{m-1}}=p^{n-m}.$$
Thus $p^{n-m} \leqslant t$ hence $|G| = p^n = p^m \cdot p^{n-m} \leqslant t^{t-1} \cdot t = t^t$.
\end{proof}

Theorem \ref{CB} is now a corollary of the following result.

\begin{thm} \label{AxBcyclic}
Let $G$ be a finite group with $c(G) = t$. Then $G \cong A \times B$ where $A$ is cyclic, $|A|,|B|$ are coprime, the Sylow subgroups of the center of $B$ have order at most $t^t$ and $|B| \leqslant \alpha^t \cdot t^{t^2}$ where $\alpha$ is the constant in \cite{Pyber}.
\end{thm}

\begin{proof}
We prove, by induction on $|G|$, that $G \cong A \times B$ where $A$ is cyclic, $|A|,|B|$ are coprime, the Sylow subgroups of the center of $B$ have order at most $t^t$ and $|B| \leqslant \alpha^t \cdot t^{t^2}$. Proposition \ref{pgroup} implies that every noncyclic $p$-subgroup of $G$ has order at most $t^t$. By Cauchy's theorem, for every prime $p$ dividing $|G|$, there exists a cyclic subgroup of $G$ of order $p$, therefore $|G|$ has at most $t$ prime divisors. So if $Z(G)$ has no Sylow subgroup of order larger than $t^t$ then $|Z(G)| \leqslant t^{t^2}$. Proposition \ref{noncommuting} now implies that $|G| \leqslant \alpha^t \cdot t^{t^2}$ and hence we can choose $A=\{1\}$. Now assume that $Z(G)$ has a Sylow subgroup $P$ with the property that $|P| > t^t$. In particular, $P$ is cyclic and of course $P \unlhd G$.

Let $Q$ be a Sylow $p$-subgroup of $G$ containing $P$ and write $|Q|=p^k$. If $Q$ is noncyclic, then $|P| \leqslant |Q| \leqslant t^t$, a contradiction. So $Q$ is cyclic, hence $Q$ has $k+1$ cyclic subgroups. It follows that $k+1 \leqslant c(G) = t$. 
If $Q$ is not normal in $G$, then $1 < n_p(G) \equiv 1 \mod p$ by Sylow's theorem, hence $p < n_p(G) \leqslant t$, therefore $|P| \leqslant |Q| = p^k \leqslant t^{t-1} < t^t$, a contradiction.
We deduce that $Q$ is cyclic and normal in $G$, so that $G = Q \rtimes K$ for a suitable $K$, by the Schur-Zassenhaus theorem. If this is not a direct product, then $K$ is not normal in $G$, so $Q$ is not contained in $N_G(K)$. Therefore $Q$ cannot normalize all cyclic subgroups of $K$, in other words there is some $u \in K$ such that $Q$ is not contained in $N_G(\langle u \rangle)$. It follows that $p$ divides the index $|G:N_G(\langle u \rangle)|$, which equals the number of conjugates of $\langle u \rangle$ in $G$, hence $p \leqslant c(G) = t$. We deduce that $|P| \leqslant |Q| = p^k \leqslant t^{t-1} < t^t$, a contradiction. This implies that $G$ is a direct product $Q \times K$ where $Q$ is a cyclic $p$-group and $p$ does not divide $|K|$. By induction, the result holds for $K$, so it holds for $G$ and we are done.
\end{proof}

\section{Proof of Theorem \ref{setB}}

We now prove Theorem \ref{setB}, which we restate here for convenience.

Call $\mathscr{B}$ the set of positive integers $n$ such that there are only finitely many noncyclic groups $G$ with $c(G)=n$, up to isomorphism. 
We will prove that $\mathscr{B} = \{1,4,6,9\} \cup \mathscr{P}$ where $\mathscr{P}$ is the set of prime numbers.

From the work of Ashrafi and Haghi \cite{AH,HA2018} (see also \cite{Kalra} and \cite{Zhou}) we immediately deduce that $3,4,5,6,7,9 \in \mathscr{B}$ and $8,10 \not \in \mathscr{B}$. We need to show that an integer $n \geqslant 10$ is prime if and only if there are only finitely many noncyclic finite groups with precisely $n$ cyclic subgroups. By Theorem \ref{AxBcyclic}, there exists a function $f$ such that, if $G$ is any noncyclic finite group with $c(G)=n$, we have $G \cong A \times B$ with $|A|,|B|$ coprime, $A$ cyclic and $|B| \leqslant f(n)$. In particular $c(G) = c(A) \cdot c(B)$ so, if $n$ is prime, the fact that $A$ is cyclic and $G$ is noncyclic implies that $A=\{1\}$, so $G=B$ and $|G| \leqslant f(n)$. Conversely, assume $n$ is not a prime number and write $n=ab$ with $a,b > 1$. Since $n \geqslant 10$ we may assume without loss of generality that $b \geqslant 4$. Then we can write $b-2 = kq$ for some prime number $q$ and some integer $k \geqslant 1$, so that $c(C_{q^k} \times C_q) = qk+2 = b$ by Lemma \ref{lambda-abelian} and, if $r$ is any prime number distinct from $q$, the noncyclic group $C_{r^{a-1}} \times C_{q^k} \times C_q$ has $ab=n$ cyclic subgroups. Since there are infinitely many such primes $r$, we obtain the result.


\begin{thebibliography}{99}
\bibitem{AH} {A. R. Ashrafi and E. Haghi. \emph{On n-cyclic groups.} Bulletin of the Malaysian Mathematical Sciences Society, 42:3233--3246, 2019.}

\bibitem{BLR} {R. Bastos, I. Lima, J. R. Rogério. \emph{Maximal covers of finite groups.} Communications in Algebra, 48(2), 691--701, 2020.}

\bibitem{Dixon} {Dixon, J. D. (1968). The solvable length of a solvable linear group.} Math. Z. 107, I51-158.


\bibitem{Gorenstein} {D. Gorenstein. \emph{Finite Groups.} AMS Chelsea Publishing, New York, 1968.}

\bibitem{HA2018} {E. Haghi and A. R. Ashrafi. \emph{On the number of cyclic subgroups in a finite group.} Southeast Asian Bulletin of Mathematics, 42: 865--873, 2018.}

\bibitem{Hungerford} {T. W. Hungerford. \emph{Algebra.} Springer New York, 1974.}

\bibitem{Kalra} {H. Kalra. \emph{Finite groups with specific number of cyclic subgroups.} Proceedings - Mathematical Sciences, 129(52):1--10, 2019.}

\bibitem{GL} {A. Lucchini and M. Garonzi. \emph{Irredundant and Minimal Covers of Finite Groups.} Communications in Algebra, 44(4): 1722--1727, 2016.}

\bibitem{Pyber} {L. Pyber. \emph{The number of pairwise noncommuting elements and the index of the centre in a finite group.}
Journal of the London Mathematical Society, s2-35(2): 287--295, 1987.}

\bibitem{Rogerio} {J. R. Rogério. \emph{A Note on Maximal Coverings of Groups.} Communications in Algebra, 42(10):4498--4508, 2014.}

\bibitem{Schenkman} {E. Schenkman. \emph{On the norm of a group.} Illinois Journal of Mathematics, 4(1):150--152, 1960.}

\bibitem{Zhou} {W. Zhou. \emph{Finite groups with small number of cyclic subgroups.} \url{http://arxiv.org/abs/1606.02431v1}.}
\end{thebibliography}
\end{document}